\documentclass[a4paper,12pt]{article}

\usepackage{amsfonts,amssymb,amsmath,amsthm}

\usepackage{multirow}

\newtheorem{theorem}{Theorem}
  \newtheorem{lemma}{Lemma}
  \newtheorem{proposition}{Proposition}

\newcommand{\R}{{\mathbb R}}

\newcommand{\Z}{{\mathbb Z}}
\newcommand{\Q}{{\mathbb Q}}
\newcommand{\C}{{\mathbb C}}

\newcommand{\OO}{\mathcal{O}}
\DeclareMathOperator{\ho}{H} % naive height
\DeclareMathOperator{\mo}{M} % Mahler measure
\DeclareMathOperator{\res}{Res} % resultant
\DeclareMathOperator{\sep}{sep}
\DeclareMathOperator{\abssep}{abs\,sep}
\usepackage{xcolor}

\begin{document}

\title{{\Large{\bf Absolute real root separation}}}

\author{Yann Bugeaud, Andrej Dujella, \\ Tomislav Pejkovi\'{c}, and Bruno Salvy}

\date{}
\maketitle

\begin{abstract}
While the separation (the minimal nonzero distance) between roots of a polynomial is a classical topic, its absolute counterpart (the minimal nonzero distance between their absolute values) does not seem to have been studied much. We present the general context and give tight bounds for the case of real roots.
\end{abstract}

% \footnotetext{
% 2010 Mathematics Subject Classification: 11C08, 12D10, 11B37.

% Key words: integer polynomials, root separation.

% }

\section{Separation and absolute separation}
The polynomials $14x^3+17x^2-13x+2 $ and $17x^3-9x^2-7x+8$ hold records in the set of polynomials with integer coefficients in $\{-20,\dots,20\}$ and degree at most~3. The first one has two roots $\alpha_1,\alpha_2$ with
\[0<|\alpha_1-\alpha_2|<0.005, \]    %%y
while the second one has two roots $\beta_1,\beta_2$ with
\[0<\big||\beta_1|-|\beta_2|\big|<0.000015.\]
Apart from those obtained by multiplying these polynomials by $-1$ or changing $x$ into $-x$ in them, no other polynomial in that set has roots satisfying any of these inequalities. 
More generally, we are interested in understanding how close can the nonzero difference between two roots (or the absolute values of two roots) of a polynomial with integer coefficients be in terms of its degree and a bound on its coefficients. 

The first quantity is classically called the \emph{root separation} of a polynomial~$P$ with at least two distinct roots and denoted by $\sep(P)$:    %%y 
\[ \sep(P):=\min_{\substack{P(\alpha)=P(\beta)=0,\\ \alpha\neq\beta}}|\alpha-\beta|. \]
By analogy, we define the \emph{absolute root separation} by
\[ \abssep(P):=\min_{\substack{P(\alpha)=P(\beta)=0,\\ |\alpha|\neq|\beta|}}\big||\alpha|-|\beta|\big|. \]
This quantity arises in the computation of the asymptotic behavior of sequences satisfying linear recurrence equations, where it governs the precision required in computing the roots of the characteristic polynomial.

The \emph{height} of a polynomial~$P$, denoted $\ho(P)$, is the maximum of the absolute values of its coefficients. In 1964, Mahler proved the following lower bound relating separation, degree and height~\cite{M64,B04}.
\begin{proposition}[Mahler, 1964] %%y 
If $\alpha, \beta$ are two roots of a separable polynomial (i.e., with no multiple roots) of degree~$d\ge2$ with integer coefficients, then
\begin{equation} \label{eq:mahler}
|\alpha-\beta|> \sqrt{3} (d+1)^{-(2d+1)/2} \max\{1,|\alpha|,|\beta|\} \ho(P)^{-d+1}.
\end{equation}
\end{proposition}
This implies in that case
\begin{equation}\label{eq:mahler2}
\sep(P)\gg \ho(P)^{-d+1},
\end{equation}
where, here and below, the constant implicit in the `$\gg$' sign depends   %%y 
only on the degree~$d$. 
The tightness of the exponent $-d+1$ of $\ho(P)$ in this inequality is still unknown. The best known upper bound on this exponent is $-(2d-1)/3$ for general~$d$ (see \cite{BD14}) and~$-2$ for $d=3$  \cite{E04,S06}, which is the only case where we thus know the bound of~\eqref{eq:mahler2} to be optimal. (Restricting further the set of polynomials under consideration so that only monic or irreducible polynomials enter the contest leads to even larger upper bounds \cite{BM10,BD11,BD14,DP11}.)

Concerning the absolute separation, not much appears to be known. A consequence of Mahler's result \eqref{eq:mahler} yields the following \cite{GS96}:
\begin{equation} \label{eq:gourdonsalvy} 
  \abssep(P)\gg \ho(P)^{-d(d^2+2d-1)/2}.
\end{equation}
The exponent of $\ho(P)$ does not seem to be the best possible, but no improvement in this general case seems to be known. However, restricting to the absolute root separation of two \emph{real} roots, our main result completely solves the problem:
\begin{theorem} \label{thm:bound} 
Let $\alpha_1=\alpha,\alpha_2=\beta,\alpha_3,\ldots,\alpha_d\in\C$ be the roots of a separable polynomial $P(x)\in\Z[x]$ of degree $d$ such that $\alpha_i+\alpha_j\neq 0$ for any $i,j\in\{1,\ldots,d\}$. If $\alpha$ and $\beta$ are real then
\begin{equation} \label{eq:bound}
  \big||\alpha|-|\beta|\big|\geq 2^{(-d^2+2)/2}(d+1)^{(-d+1)/2}\ho(P)^{-d+1}. 
\end{equation}
Moreover, the exponent of $\ho(P)$ is best possible.
\end{theorem}
In the next section, we give a proof of the first part of this theorem, which is a variant of Mahler's proof. 
Next, we prove that for reducible polynomials, the exponent can be increased to $-d+2$ (Theorem~\ref{thm:reducible} below).
Finally, we give explicit families of polynomials reaching the bound $-d+1$. 
We show that, %%y 
furthermore, the bound is still tight when restricting to monic polynomials of arbitrary degree $d\ge3$.

\section{Bounds from resultants}\label{sec2}

Since Mahler's result \eqref{eq:mahler} implies \eqref{eq:bound} if $\alpha$ and $\beta$ have the same sign, we need to prove a comparable result for $|\alpha+\beta|$ as well. Our proof, like Mahler's, relies on the use of a resultant. Recall that the resultant~$\res(P,Q)$ of two univariate polynomials~$P$ and~$Q$ of degree~$d_1$ and~$d_2$ is the determinant of their Sylvester matrix, which is the transposed matrix of the map $(U,V)\mapsto UP+VQ$ in the bases $((1,0),(x,0),\dots,(x^{d_2-1},0),(0,1),(0,x),\dots,(0,x^{d_1-1}))$ and $(1,x,\dots,x^{d_1+d_2-1})$. We refer to the literature~(e.g., \cite{L02,M92}) for the definition and properties of resultants and recall only the following.
\begin{lemma}[Poisson's formula] 
The resultant of the polynomials\\ $P=a_{d_1}\prod_{i=1}^{d_1}(x-\alpha_i)$ and      %%y 
$Q=b_{d_2}\prod_{i=1}^{d_2}(x-\beta_i)$ satisfies     %%y 
\[\res(P,Q)=a_{d_1}^{d_2}b_{d_2}^{d_1}\prod_{i,j}(\alpha_i-\beta_j).\]
\end{lemma}
The special case when $Q=P'$ gives the discriminant (up to a simple constant), that was used by Mahler in his proof.

We thus obtain the products of pairwise sums of roots of $P(x)=a_d x^d+\dots+a_1 x+a_0$ by considering the resultant $r=\res(P(x),P(-x))$,
\[ r=
  \begin{vmatrix}
    a_d & \cdots & a_1 & a_0  &   &   \\
      & \ddots &   &   & \ddots &   \\
      &   &   a_d & a_{d-1} & \cdots & a_0 \\
    (-1)^{d} a_d & \cdots & -a_1 & a_0  &   &   \\
      & \ddots &   &   & \ddots &   \\
      &   &   (-1)^{d} a_d & (-1)^{d-1} a_{d-1} & \cdots & a_0 \\
  \end{vmatrix}.\]
Factoring out $a_d$ from the first column and $a_0$ from the last, we see that $r$ is an integer divisible by $a_0 a_d$. Thus, if $r\neq 0$, then we have 
\begin{equation} \label{eq:rlower}
  |r|\geq |a_0 a_d|.
\end{equation}
On the other hand, Poisson's formula gives 
\begin{equation*} \label{eq:rproduct}
\begin{aligned}
r 
  &=(-1)^d a_{d}^{2d}  \prod_{1\leq i\leq d}(2\alpha_i) \, \prod_{1\leq i< j\leq d} (\alpha_i+\alpha_j)^2,\\
  &=a_0a_{d}^{2d-1}  2^d \prod_{1\leq i< j\leq d} (\alpha_i+\alpha_j)^2.
\end{aligned}
\end{equation*}
since $\prod_{i=1}^{d} \alpha_i= (-1)^d a_0/a_d$.

Bounds will follow from the simple upper bound %%y 
\begin{equation*} \label{eq:tworoots}
  |\alpha_i+\alpha_j|\leq 2\max\{1,|\alpha_i|\}\max\{1,|\alpha_j|\},\quad\text{for all $i,j$}.
\end{equation*}
Then, with our notation $\alpha=\alpha_1$, $\beta=\alpha_2$, we get
\begin{equation} \label{eq:rupper}
  |r| 
 \leq  2^{d^2-2} |a_0||a_d|^{2d-1} {|\alpha+\beta|^2} 
    \prod_{1\leq i\leq d} \max\{1,|\alpha_i|\}^{2d-2}.
\end{equation}
This last quantity is related to the height of $P$ thanks to the following.
\begin{lemma}[Landau's inequality]
The \emph{Mahler measure} $\mo(P)$, defined as
\[ \mo(P):=|a_d|\prod_{i=1}^{d}\max\{1,|\alpha_i|\}, \]
is bounded by:
\begin{equation} \label{eq:mahlermeasure}
\mo(P)\leq \sqrt{a_0^2+\dots+a_d^2}\le \sqrt{d+1}\ho(P).
\end{equation}
\end{lemma}
\begin{proof}[Sketch of proof]
See, for instance, \cite{B04,M92}. If all the roots have modulus smaller than~1, then the property is obvious. If $P(\alpha)=0$ with~$|\alpha|>1$ then the polynomial $(1-\alpha x)P(x)/(x-\alpha)$ can be checked to have the same Mahler measure and  %%y 
the same $L_2$-norm as $P$.   %%y 
This construction is repeated till no root has modulus larger than~1.
\end{proof}
When no $\alpha_i+\alpha_j$ is~0, then Landau's bound~\eqref{eq:mahlermeasure} together with~\eqref{eq:rupper} and~\eqref{eq:rlower} give
\[  |\alpha+\beta|\geq 2^{(-d^2+2)/2}(d+1)^{(-d+1)/2}\ho(P)^{-d+1},  \]
which finishes the proof of the first part of Theorem~\ref{thm:bound}.

\section{The case of reducible polynomials}
It turns out that the exponent of $\ho(P)$ can be improved for irrational roots of reducible polynomials.
\begin{theorem} \label{thm:reducible}
Let $P(x)\in\Z[x]$ be a polynomial of degree~$d$ reducible over $\Q$. Then for any two irrational real roots $\alpha,\beta$ of $P(x)$, either $|\alpha|=|\beta|$ or 
\begin{equation} \label{eq:reducible}
  \big| |\alpha|-|\beta| \big| \ge 2^{-3d^2/2+3d-1} d^{(-d+2)/2} \ho(P)^{-d+2}.
\end{equation}
\end{theorem}
\begin{proof} We assume $|\alpha|\neq|\beta|$. Any irreducible factor~$Q$ of $P$ has degree at most $d-1$ and Gelfond's lemma (see e.g., \cite[Lemma A.3]{B04}) implies $\ho(Q)\leq 2^d \ho(P)$. 

We distinguish two cases depending on whether $\alpha$ and $\pm\beta$ are roots of the same irreducible factor of~$P$ or not. When $\alpha$ and either one of $\pm\beta$ are roots of the same irreducible factor~$Q$, Theorem~\ref{thm:bound} applies to~$Q$ if moreover it has no pair of opposite roots. Then, we get
\begin{equation} \label{eq:notevenpoly}
\begin{aligned}
    \big| |\alpha|-|\beta| \big| &\ge 2^{(-d^2+2d+1)/2} d^{(-d+2)/2} \ho(Q)^{-d+2}\\
    &\ge 2^{-3d^2/2+3d+1/2} d^{(-d+2)/2} \ho(P)^{-d+2}.
\end{aligned}
\end{equation}
If two roots of $Q$ sum to~0, the gcd of $Q(x)$ and $Q(-x)$ is nontrivial and since $Q$ is irreducible, this means that $Q(x)$ is even, so that $\pm\beta$ are both roots of $Q$ and Mahler's bound applies both to~$|\alpha-\beta|$ and $|\alpha+\beta|$, leading to
\begin{equation} \label{eq:evenpoly}
  \big| |\alpha|-|\beta| \big| = |\alpha - (\pm \beta)| > \sqrt{3} d^{-(2d-1)/2} \ho(Q)^{-d+2},
\end{equation} 
which leads to the result since $\sqrt{3}d^{-(2d-1)/2}>2^{(-d^2+2d-2)/2}d^{(-d+2)/2}$  for $d\ge0$.

Otherwise, if $\alpha$ and $\pm\beta$ are not roots of the same irreducible factor of~$P$,
let $Q_{\alpha}(x)$ and $Q_{\beta}(x)$ be minimal polynomials of $\alpha$ and $\beta$ over $\Z$ with respective degrees $a$ and $b$.
Arguing in a similar manner as in the proof of Theorem \ref{thm:bound},
we deduce
\[1\le|\res(Q_{\alpha}(x),Q_{\beta}(\pm x))|
\leq 2^{ab} \mo(Q_{\alpha})^{b} \mo(Q_{\beta})^{a} |\alpha\pm\beta|.
\]
Gauss's lemma and multiplicativity of Mahler's measure give
$ \mo(Q_{\alpha}Q_{\beta}) \leq \mo(P)$. Furthermore, 
the condition that $\alpha$ and $\beta$ are not rational gives $a,b\in [2,d-2]$, so that
\[ \mo(Q_{\alpha})^{b} \mo(Q_{\beta})^{a} \leq \mo(Q_{\alpha}Q_{\beta})^{\max\{a,b\}}
  \leq \mo(P)^{d-2} \leq (d+1)^{(d-2)/2} \ho(P)^{d-2} \]
and finally $2^{ab}\leq 2^{d^2/4}$, leading to
\[  \big| |\alpha|-|\beta| \big| \geq 2^{-d^2/4} (d+1)^{(-d+2)/2} \ho(P)^{-d+2} \]
which is stronger than \eqref{eq:reducible} since $\alpha,\beta$ being both irrational forces $d\ge4$.
\end{proof}

\section{Families with small absolute root separation}
We now exhibit families of polynomials that reach the exponent~$-d+1$ of Theorem~\ref{thm:bound}, thereby concluding its proof. The construction starts from~$Mx^2-1$ which has two real roots $\pm1/\sqrt{M}$. For $M$ sufficiently large, these are small so that higher powers of the roots will be even smaller. By perturbing $Mx^2-1$ in appropriate ways we thus get polynomials with roots very close to~$\pm1/\sqrt{M}$ and whose sum is minute.
\begin{theorem} \label{thm:examples}
Let $d \ge 2$ be an integer and $M$ be a positive integer.
Consider the polynomials $P_{d,M}(x)$ of degree $d$ and height $M$  
defined by:
\begin{equation*}
P_{d,M}(x) =
 \begin{cases} 
   Mx^2-x-1 & \text{if } d=2; \\
   x^3+Mx^2-1 &\text{if } d=3; \\
   x^d-(Mx^2-1)(1-x^{d-3}) &\text{if } d\geq 4 \text{ even;} \\
   x^d -(M-1)x^{d-1} +x^{d-3} -Mx^3 -Mx^2 +x +1 &\text{if } d\geq 5 \text{ odd.}
 \end{cases}
\end{equation*}
If $M$ is sufficiently large in terms of $d$, 
then the polynomial $P_{d,M}(x)$ has two roots $\alpha,\beta\in\R$ such that $||\alpha|-|\beta||=|\alpha+\beta|$ and
\begin{equation} \label{eq:polybound}
  0 < |\alpha+\beta| \ll \ho(P_{d,M})^{-d+1}.
\end{equation}
For $d\geq 3$, these polynomials are monic and irreducible over the field of rational numbers.
\end{theorem}

Note that for any odd integer $d\geq 5$ the polynomial $x^{d-1}-(Mx^2-1)(1-x^{d-2})$ has also two real roots satisfying \eqref{eq:polybound}.
However, $P_{d,M}$ has the extra advantage of being monic.

\begin{proof}[Proof of Theorem \ref{thm:examples}] The case of quadratic polynomials is trivial. The only thing left to discuss for $d=2$ is monicity and irreducibility. With monic quadratic polynomials, we can only achieve the exponent $0$ in \eqref{eq:polybound}. The given polynomial is irreducible if and only if $4M+1$ is not a square,    %%y 
which happens for infinitely many $M$. Therefore, from now on, we take $d\geq 3$.

Direct asymptotic computations give roots $\alpha$ and $\beta$ of~$P_{d,M}$ such that \eqref{eq:polybound} holds. For instance, 
the polynomial $P_{3,M}$ has two roots $\alpha_3, \beta_3$ satisfying %%y 
\begin{align*}
  \alpha_3 &= -M^{-1/2} -\frac{1}{2}M^{-2} +\OO(M^{-7/2}), \\ %-\frac{5}{8}M^{-7/2} -M^{-5} -\frac{231}{128}M^{-13/2} +\OO(M^{-8}) \\
  \beta_3 &= M^{-1/2} -\frac{1}{2}M^{-2} +\OO(M^{-7/2}). %\frac{5}{8}M^{-7/2} -M^{-5} +\frac{231}{128}M^{-13/2} +\OO(M^{-8})  \\
\end{align*}
The conclusion of the proof is computational too and consists in giving small enough intervals where the polynomial $P_{d,M}(x)$ changes sign. For the first root for instance,
\[ \begin{aligned}
  P_{3,M}(-M^{-1/2}-\frac{1}{2}M^{-2}-M^{-3}) &= 2M^{-5/2} +\OO(M^{-3}) >0, \\
  P_{3,M}(-M^{-1/2}-\frac{1}{2}M^{-2}) &= -\frac{5}{4}M^{-3} +\OO(M^{-4}) <0,
\end{aligned} \]
for $M$ large enough and similarly for the other root. Of course, the lengths of the intervals can be adjusted to the required precision.

Correspondingly, for $d\geq 4$ even, $P_{d,M}(x)$ has roots
\[ \begin{aligned}
  \alpha &= -M^{-1/2} -\frac{1}{2}M^{-(d+1)/2} +\frac{1}{2}M^{-d+1}+\OO(M^{-(2d+1)/2}),\\ % -\frac{2d-1}{8}M^{-(2d+1)/2} 
%  \\ &\qquad -\frac{1}{2}M^{-(3d-5)/2} +\OO(M^{-(3d-2)/2}) \\
  \beta &= M^{-1/2} +\frac{1}{2}M^{-(d+1)/2} +\frac{1}{2}M^{-d+1} +\OO(M^{-(2d+1)/2}).\\ %+\frac{2d-1}{8}M^{-(2d+1)/2} 
%  \\ &\qquad +\frac{1}{2}M^{-(3d-5)/2} +\OO(M^{-(3d-2)/2})
\end{aligned} \]
%and thus we get $|\alpha+\beta|=M^{-d+1} +\OO(M^{-(3d-2)/2})$.
For $d\geq 5$ odd, $P_{d,M}(x)$ has roots
\[ \begin{aligned}
  \alpha &= -M^{-1/2} -\frac{1}{2}M^{-d/2} +\frac{1}{2}M^{-(2d-3)/2} +\frac{1}{2}M^{-d+1} + \OO(M^{-(2d-1)/2}),\\ 
%       -\frac{2d-7}{8}M^{-(2d-1)/2} \\ &\qquad +\frac{1}{2}M^{-d} +\OO(M^{-(2d+1)/2}) \\
  \beta &= M^{-1/2} +\frac{1}{2}M^{-d/2} -\frac{1}{2}M^{-(2d-3)/2} +\frac{1}{2}M^{-d+1} 
   + \OO(M^{-(2d-1)/2}).
   % +\frac{2d-7}{8}M^{-(2d-1)/2} \\ &\qquad +\frac{1}{2}M^{-d} +\OO(M^{-(2d+1)/2})
\end{aligned} \]
%and thus we get $|\alpha+\beta|=M^{-d+1} +\OO(M^{-d})$.
Inequalities showing that actual roots with these expansions exist are obtained by straightforward computations.

For $d \ge 3$ and $M \ge 3$, the leading coefficient of the polynomial $P_{d,M}(x)$ is $1$ and its 
constant coefficient is $\pm 1$. It is easily checked that $1$ and $-1$ are not roots of $P_{d,M}(x)$, thus this polynomial has no rational roots. Gauss' lemma together with Theorem \ref{thm:reducible} then implies that $P_{d,M}(x)$ is irreducible for $M$ large enough in terms of $d$.
\end{proof}

\section*{Acknowledgements}
A.~D. and T.~P. were supported by the Croatian Science Foundation under the project no.~6422. A.~D. acknowledges support from the QuantiXLie Center of Excellence. The work of B.~S. was supported in part by FastRelax ANR-14-CE25-0018-01.

\bigskip

{\small \noindent
Yann Bugeaud \\
Universit\'{e} de Strasbourg \\ 
Math\'{e}matiques \\ 
7, rue Ren\'{e} Descartes \\
F-67084 Strasbourg cedex, France \\
{\em E-mail address}: {\tt bugeaud@math.unistra.fr}}

\bigskip 

{\small \noindent
Andrej Dujella \\
Department of Mathematics \\ University of
Zagreb
\\ Bijeni\v{c}ka cesta 30 \\
10000 Zagreb, Croatia \\
{\em E-mail address}: {\tt duje@math.hr}}

\bigskip

{\small \noindent
Tomislav Pejkovi\'c \\
Department of Mathematics \\ University of
Zagreb
\\ Bijeni\v{c}ka cesta 30 \\
10000 Zagreb, Croatia \\
{\em E-mail address}: {\tt pejkovic@math.hr}}

\bigskip

{\small \noindent
Bruno Salvy \\
INRIA and LIP -- ENS Lyon \\
46, all\'{e}e d'Italie\\
69364 Lyon Cedex 07, France \\
{\em E-mail address}: {\tt Bruno.Salvy@inria.fr}}

\end{document}